\documentclass[11pt]{article}
\usepackage{amssymb}
\usepackage{amsmath}
\usepackage{amsfonts}
\usepackage{amsthm}
\usepackage{enumerate}

\numberwithin{equation}{section}

\newtheorem{theorem}{Theorem}[section]
\newtheorem{lemma}{Lemma}[section]
\newtheorem{proposition}{Proposition}[section]

\newcommand\re[1]{(\ref{#1})}
\def\R{\mathbb{R}}
\def\Z{\mathbb{Z}}

\def\S{\mathcal{S}}

\def\N{\mathbb{N}}
\def\NN{\mathcal{N}}
\def\ZZ{\mathcal{Z}}
\def\F{\mathcal{F}}
\def\L{\mathcal{L}}
\def\T{\mathbb{T}}
\def\eps{\varepsilon}
\def\supp{\mathop{\rm supp}\nolimits}

\newcommand\cro[1]{\langle #1 \rangle}

\begin{document}
\begin{center}
\noindent {\Large \bf{Sharp ill-posedness and well-posedness results for the KdV-Burgers equation: the periodic case}}
\end{center}
\vskip0.2cm
\begin{center}
\noindent
{\bf    Luc Molinet and St\'ephane Vento }\\
\end{center}
\vskip0.5cm \noindent {\bf Abstract.} { We prove that the KdV-Burgers is globally well-posed in $ H^{-1}(\T) $  with a solution-map that is analytic from $H^{-1}(\T) $ to  $C([0,T];H^{-1}(\T))$ whereas it is ill-posed
   in $ H^s(\T) $, as soon as $ s<-1 $, in the sense that the flow-map $u_0\mapsto u(t) $ cannot be continuous
    from $ H^s(\T) $ to even ${\cal D}'(\T) $ at any fixed $ t>0 $ small enough.  In view of the result of Kappeler and Topalov for KdV  it thus appears that even if the dissipation part of the KdV-Burgers equation
  allows to lower the $ C^\infty $ critical index with respect to the KdV equation,
       it does not  permit to improve the
      $ C^0$ critical index .
}\vspace*{4mm} \\
\vskip0.2cm

\section{Introduction and main results}
The aim of this paper is to establish positive and negative optimal results on the Cauchy problem in Sobolev spaces for the Korteweg-de Vries-Burgers (KdV-B) equation posed on the one dimensional torus $ \T=\R/2\pi \Z$:
\begin{equation}\label{KdVB}
u_t +u_{xxx}-u_{xx}+u u_x =0
\end{equation}
where $ u=u(t,x) $ is a real valued function.\\
This  equation has been derived by Ott and Sudan \cite{OS}  as an asymptotic  model for the propagation  of weakly nonlinear dispersive long waves in some physical contexts when dissipative effects occur.

In order to make our result more transparent,  let us first introduce   different notions of well-posedness (and consequently ill-posedness) related to the smoothness of the flow-map  (see in the same spirit \cite{KT}, \cite{Gerard}).
 Throughout this paper we shall say that  a Cauchy problem is  (locally)  $ C^0$-well-posed   in some  normed
  function space $ X $  if, for any initial data $ u_0\in X $, there exist a radius $ R>0 $, a  time
  $ T>0 $  and a unique solution $ u$,  belonging to some space-time function space
  continuously embedded in $ C([0,T];X) $, such that for any $ t\in [0,T] $ the map
   $ u_0\mapsto u(t) $ is continuous from the ball of $ X $ centered at $ u_0 $ with radius $ R $
    into $ X $.
    If   the map $ u_0\mapsto u(t) $ is of class $ C^k $, $k\in \N\cup\{\infty\} $, (resp. analytic) we will say that the Cauchy is $ C^k$-well-posed (resp. analytically well-posed).   Finally  a Cauchy problem will be  said to be $C^k $-ill-posed, $k\in  \N\cup\{\infty\} $,  if it is not $C^k $-well-posed.

 In \cite{MR2}, Molinet and Ribaud   proved that this  equation is analytically  well-posed in $ H^s(\T) $ as soon as $ s>-1 $. They also established that the index $ -1 $ is critical for
 the  $C^2 $-well-posedness.  The surprising part of this result was that the $ C^\infty $ critical index $ s_c^\infty(KdVB)=-1 $
   was lower that the one of  the KdV equation
   $$
   u_t +u_{xxx}+u u_x =0
   $$
   for which $s_c^\infty(KdV)=-1/2 $  (cf. \cite{KPV}, \cite{Co})  and also lower than the $ C^\infty $ index $ s_c^\infty(dB)=s_c^0(dB)=-1/2 $ (cf. \cite{Be}, \cite{Di1}) of the dissipative Burgers
   equation
   $$
   u_t -u_{xx}+u u_x =0\; .
   $$
    On the other hand, using the integrability theory, it was recently proved in \cite{KT} that the flow-map of KdV equation can be uniquely
  continuously extended in $H^{-1}(\T) $.  Therefore, on the torus,  KdV  is $ C^0 $-well-posed in $ H^{-1} $ if one takes as
   uniqueness class, the class of strong  limit in $C([0,T];H^{-1}(\T)) $ of smooth solutions.

   In \cite{MV1} the authors completed the result of \cite{MR2} in the real line case by
    proving that the KdV-Burgers equation
 is analytically well-posed in $ H^{-1}(\R)$  and $ C^0 $-ill-posed in
  $H^s(\R) $ for $ s<-1 $ in the sense that the flow-map defined on $ H^{-1}(\R)$
   is not continuous for the topology inducted by $ H^s $, $s<-1 $, with values even in
    ${\cal D}' (\R)$.
    To reach the critical Sobolev space $ H^{-1} (\R)$  they adapted the
      refinement of Bourgain's spaces that appeared in \cite{tataru} and
      \cite{Tao} to the framework developed in \cite{MR2}. The proof of the main bilinear estimate used in a crucial way the Kato smoothing effect that does not hold on the torus. Our aim here is to give the new ingredients that enable to overcome this lack of smoothing effects. The main idea is to weaken the space regularity of the Bourgain's spaces
         in a suitable space-time frequencies region. Note that  our resolution space will still be embedded in
        $ C([0,T];H^{-1}(\T)) $ and that, to get  the  $ L^\infty([0,T];H^{-1}(\T)) $-estimate  in this region,  we use an  idea that appeared in \cite{Bo3}. Finally, once the well-posedness result is proved, the proof of the ill-posedness result follows exactly the same lines as in \cite{MV1}. It  is due to a high to low frequency  cascade phenomena that was first observed in  \cite{BT} for a quadratic Schr\"odinger equation.

      In view of the result of Kappeler and Topalov for KdV  it thus appears that, at least on the torus, even if the dissipation part of the KdV-Burgers equation\footnotemark[1]
 allows to lower the $ C^\infty $ critical index with respect to the KdV equation,
       it does not  permit to improve the
      $ C^0$ critical index .
\footnotetext[1]{It is important to notice that
 the dissipative term $ -u_{xx} $ is of lower order than the dispersive one $ u_{xxx}$.}

  Our results can be summarized as follows:
\begin{theorem}\label{wellposed}
 The Cauchy problem associated to  (\ref{KdVB}) is  locally analytically  well-posed in $ H^{-1}(\T) $.
Moreover,  at every point $ u_0 \in  H^{-1}(\T) $ there exist $ T=T(u_0)>0 $ and $R=R(u_0)>0 $ such that  the solution-map $ u_0\mapsto u  $ is analytic from the ball centered at $ u_0 $ with radius $R$ of $ H^{-1}(\T) $ into $ C([0,T];H^{-1}(\T)) $. Finally, the solution  $ u$ can be extended for all positive
  times  and belongs to $
 C(\R_+^*;H^\infty(\T)) $.
\end{theorem}
Now that we have established analytic well-posedness, proceeding exactly as in \cite{MV1} by taking as sequence of initial data
  $$
  \phi_N(x)=N \cos(Nx) \; ,
  $$
  we get the following ill-posedness result.
\begin{theorem}\label{illposed}
The Cauchy problem associated to  (\ref{KdVB}) is ill-posed in $ H^s(\T) $ for $s <-1 $ in the following sense: there exists $ T>0 $ such that for any $ 0<t<T $,  the flow-map $ u_0\mapsto u(t) $ constructed in Theorem \ref{wellposed} is discontinuous at the origin from $ H^{-1}(\T) $  endowed with the topology inducted by $ H^s(\T) $ into ${\mathcal D}'(\T) $.
\end{theorem}
\noindent
{\bf Acknowledgements:}   L.M.  was partially supported by the ANR project
 "Equa-Disp".

 \section{Resolution space}\label{sec-space}
In this section we introduce a few notation and we define our functional framework.

For $A,B>0$, $A\lesssim B$ means that there exists $c>0$ such that $A\leq cB$. When $c$ is a small constant we use $A\ll B$. We write $A\sim B$ to denote the statement that $A\lesssim B\lesssim A$.
For $u=u(t,x)\in\S'(\R\times\T)$, we denote by $\widehat{u}$ (or $\F_xu)$ its Fourier transform in space, and $\widetilde{u}$ (or $\F u$) the space-time Fourier transform of $u$. We consider the usual Lebesgue spaces $L^p$, $L^p_xL^q_t$ and abbreviate $L^p_xL^p_t$ as $L^p$. Let us define the Japanese bracket $\cro{x}=(1+|x|^2)^{1/2}$ so that the standard non-homogeneous Sobolev spaces are endowed with the norm $\|f\|_{H^s}=\|\cro{\nabla}^sf\|_{L^2}$.

We use a Littlewood-Paley analysis. Let $\eta\in C^\infty_0(\R)$ be such that $\eta\geq 0$, $\supp \eta\subset [-2,2]$, $\eta\equiv 1$ on $[-1,1]$. We define next $\varphi(k)=\eta(k)-\eta(2k)$.
Any summations over capitalized variables such as $N,L$ are presumed to be dyadic, i.e. these variables range over numbers of the form $2^\ell$, $\ell\in \N\cup \{-1\}$. We set $ \varphi_{\frac{1}{2}}\equiv \eta $ and for $ N\ge 1 $, $\varphi_N(k)=\varphi(k/N)$ and define the operator
$P_N$ by $\F(P_Nu)=\varphi_N \widehat{u}$. We introduce $\psi_L(\tau,k)=\varphi_L(\tau-k^3)$ and for any $u\in\S'(\R\times \T)$,
$$\F_x(P_Nu(t))(k)=\varphi_N(k)\hat{u}(t,k),\quad \F(Q_Lu)(\tau,k)=\psi_L(\tau,k)\tilde{u}(\tau,k).$$
Roughly speaking, the operator $ P_{1/2} $ and $Q_{1/2} $  localize respectively in the ball  $\{|k|\lesssim 1\}$ and $ \{|\tau-k^3|\lesssim 1\}$
 whereas for $  N\ge  1$, the operator $P_N$ localizes in the annulus $\{|k|\sim N\}$ and $Q_N$ localizes in the region $\{|\tau-k^3|\sim N\}$.

Furthermore we define more general projection $P_{\lesssim N}=\sum_{N_1\lesssim N}P_{N_1}$, $Q_{\gg L}=\sum_{L_1\gg L}Q_{L_1}$ etc.

Let $e^{-t\partial_{xxx}}$ be the propagator associated to the Airy equation and define the two parameters linear operator $W$ by
\begin{equation}
W(t,t')\phi=\sum_{k\in \Z} \exp(itk^3-|t'|k^2)\hat{\phi}(k)\,  e^{ikx}  ,\quad t\in\R.\label{trtr}
\end{equation}
The operator $W : t\mapsto W(t,t)  $ is clearly an extension on $ \R $ of the linear semi-group $ S(\cdot) $ associated with \re{KdVB} that is given by
\begin{equation}
S(t)\phi=\sum_{k\in\Z} \exp(itk^3-t k^2)\hat{\phi}(k) \,  e^{ikx} ,\quad t\in\R_+.\label{trtr2}
\end{equation}
We will mainly work on the integral formulation of (\ref{KdVB}):
\begin{equation}\label{duhamel}
u(t) = S(t)u_0-\frac 12\int_0^tS(t-t')\partial_xu^2(t')dt',\quad t\in\R_+.
\end{equation}
Actually, to prove the local existence result, we will follow the strategy of \cite{MV1} and apply a fixed point argument to the following extension
 of (\ref{duhamel}):
\begin{eqnarray}\label{eq-int}u(t)& =&\eta(t)\Bigl[W(t)u_0-\frac 12 \chi_{\R_+}(t)\int_0^tW(t-t',t-t')\partial_xu^2(t')dt'
\nonumber \\
&&\hspace*{20mm}-\frac 12 \chi_{\R_-}(t)\int_0^tW(t-t',t+t')\partial_xu^2(t')dt'\Bigr] .
\label{eq-int}\end{eqnarray}
It is clear that if $u$ solves (\ref{eq-int}) then $u$ is a solution of (\ref{duhamel}) on $[0, T]$, $T<1$.

In \cite{MV1}, adapting some ideas of \cite{tataru} and \cite{Tao} to the framework developed in \cite{MR2},  the authors performed the iteration
 process in the  sum space $X^{-1,\frac 12,1}+Y^{-1,\frac 12} $ where
$$\|u\|_{X^{s,b,1}}=\Big(\sum_N\Big[\sum_L \cro{N}^{s}\cro{L+N^2}^{b}\|P_NQ_Lu\|_{L^2_{xt}}\Big]^{2}\Big)^{1/2}.$$
and
$$
\|u\|_{Y^{s,b}} =\Big(\sum_N [\cro{N}^s\|\F^{-1}[(i(\tau-k^3)+k^2+1)^{b+1/2}\varphi_N \widetilde{u}]\|_{L^1_tL^2_x}]^2\Big)^{1/2},
$$
so that
$$\|u\|_{Y^{-1,\frac 12}} \sim \Big(\sum_N[\cro{N}^{-1}\|(\partial_t+\partial_{xxx}-\partial_{xx}+I)P_Nu\|_{L^1_tL^2_x}]^2\Big)^{1/2}.$$
As explained in the introduction, due to the lack of the Kato smoothing effect on the torus,  we will be able to control none of the two above norms
 in the region "$ \sigma $-dominant". The idea is then to weaken the required
$ x$ -regularity   on the  $ X^{s,b} $ component of our resolution  space in this region.
For $\eps>0$ small enough, we thus introduce the function space
 $  X^{s,b, q}_\varepsilon  $ endowed with the norm
\begin{align}
\|u\|_{X^{s,b, 1}_\varepsilon } &= \Big(\sum_N\Big[\sum_{L\le N^3}  \cro{N}^{s}\cro{L+N^2}^{b}\|P_NQ_Lu\|_{L^2_{xt}}\Big]^{2}\Big)^{1/2}\nonumber \\
& \quad +  \Big(\sum_N\Big[\sum_{L> N^3}  \cro{N}^{s-\varepsilon}\cro{L+N^2}^{b}\|P_NQ_Lu\|_{L^2_{xt}}\Big]^{2}\Big)^{1/2}.\label{newnorm}
\end{align}
However,  $ X^{s,b, 1}_\varepsilon$ is  not embedded  anymore in  $ L^\infty(\R;H^{-1}(\T))$. For this reason
  we will take   its intersection with the function space $ \widetilde{L^\infty_{t }H^{-1}}$,   that is a  dyadic version    of
   $ L^\infty(\R;H^{-1}(\T))$,  equipped with the norm
$$\|u\|_{ \widetilde{L^\infty_{t }H^{-1}}} = \Big(\sum_N[\cro{N}^{-1}\|P_Nu\|_{L^\infty_tL^2_x}]^2\Big)^{1/2}.$$

 Finally, we also need to define  the space $ Z^{s,-\frac 12} $ equipped with the norm
  $$
 \|u\|_{Z^{s,-\frac 12}} =\Big(\sum_N [\cro{N}^s\| \varphi_N(k)\cro{i(\tau-k^3)+k^2}^{-1} \widetilde{u} \|_{L^2_k L^1_\tau}]^2\Big)^{1/2}.
 $$
We are now in position to form our resolution space $\widetilde{\S^{s}_\varepsilon}=(X^{s,\frac 12,1}_\varepsilon\cap
  \widetilde{L^\infty_{t }H^{-1}})+Y^{s,\frac 12}$ and the "nonlinear space" $\NN^{s}_\varepsilon=(X^{s,-\frac 12,1}_\varepsilon\cap Z^{s,-\frac 12})+Y^{s,-\frac 12}$   where the nonlinear term $\partial_xu^2 $ will take place.
   Actually  we will estimate $\|\partial_xu^2 \|_{\NN^{s}_\varepsilon}$ in term of   $\|u\|_{\S^{s}_\varepsilon}
    $ where   $ \S^{s}_\varepsilon=X^{s,\frac 12,1}_\varepsilon+Y^{s,\frac 12} $. Obviously
 $\|u\|_{\S^{s}_\varepsilon}\le \|u\|_{\widetilde{\S^{s}_\varepsilon}}
 $ and the first of these norms   has the advantage to only see the size of the modulus of the space-time Fourier transform of the function. This will be useful when dealing with the dual form of the main bilinear estimate.

Note that we endow these sum spaces with the usual norms:
$$\|u\|_{X+Y}=\inf\{\|u_1\|_X+\|u_2\|_Y: u_1\in X, u_2\in Y, u=u_1+u_2\}.$$

In the rest of this section, we study some basic properties of  the function space $\S^{-1}_\eps$.

\begin{lemma}\label{lem-Xinfty}
For any $\phi\in L^2$,
$$\Big(\sum_L[L^{1/2}\|Q_L(e^{-t\partial_{xxx}}\phi)\|_{L^2}]^2\Big)^{1/2}\lesssim \|\phi\|_{L^2}.$$
\end{lemma}
\begin{proof}
From Plancherel theorem, we have
$$\Big(\sum_L[L^{1/2}\|Q_L(e^{-t\partial_{xxx}}\phi)\|_{L^2}]^2\Big)^{1/2}\sim \||\tau-k^3|^{1/2}\F(e^{-t\partial_{xxx}}\phi)\|_{L^2}.$$
Moreover if we set $\eta_T(t)=\eta(t/T)$ for $T>0$, then
$$\F(\eta_T(t)e^{-t\partial_{xxx}}\phi)(\tau,k) = \widehat{\eta_T}(\tau-k^3)\widehat{\phi}(k).$$
Thus we obtain with the changes of variables $\tau-k^3\to \tau '$ and $T\tau'\to\sigma$ that
$$\||\tau-k^3|^{1/2}\F(\eta_T(t)e^{-t\partial_{xxx}}\phi)\|_{L^2}\lesssim \|\phi\|_{L^2}\||\tau'|^{1/2}T\widehat{\eta}(T\tau')\|_{L^2_{\tau'}}\lesssim \|\phi\|_{L^2}.$$
Taking the limit $T\to\infty$, this completes the proof.
\end{proof}

\begin{lemma}\label{lem-Sbounds}
\begin{enumerate}
\item For  any $\eps\ge 0 $ and all  $ u\in \widetilde{\S^{-1}_\eps} $ , we have
\begin{equation}\label{est-L2S-11}
\| u\|_{L^\infty_{t }H^{-1}(\T)}\lesssim \| u\|_{\widetilde{\S^{-1}_\eps} }\quad \mbox{ and  }\quad
\Bigr(\sum_{N}\|P_{N} Q_{\le N^3 }u\|_{L^\infty_t  H^{-1}(\T)}^2 \Bigl)^{1/2}\lesssim \|u\|_{\S^{-1}_\eps} .\end{equation}

\item For any $ 0\le \eps\le 1/2 $ and all  $ u\in \S^{-1}_\eps $ , we have
\begin{equation}\label{est-L2S-1}\| u\|_{L^2_{xt}}\lesssim \| u\|_{\S^{-1}_\eps} .\end{equation}
\item For all $u\in Y^{0,1/2}$,
\begin{equation}\label{est-L2l2}\Big(\sum_L[L^{1/2}\|Q_Lu\|_{L^2}]^2\Big)^{1/2}\lesssim \|u\|_{Y^{0,\frac 12}}.\end{equation}
\end{enumerate}
\end{lemma}

\begin{proof}
\begin{enumerate}
\item
First  it is fairly obvious that $ \widetilde{L^\infty_{t}H^{-1}_x} \hookrightarrow L^\infty_{t}H^{-1}_x $ and that, according to  the
  definition of $ X^{-1,\frac 12,1}_{\eps}$,
  $$
  \Bigr(\sum_{N}\|P_{N} Q_{\le N^3 }u\|_{L^\infty_t  H^{-1}(\T)}^2 \Bigl)^{1/2}\lesssim \|u\|_{X^{-1,\frac 12,1}_{\eps}}\; .
  $$
Second, for any dyadic $ N$,
\begin{align*}
\|P_{N }u\|_{L^\infty_{t}  H^{-1}} &= \Big\|\F_t^{-1}\Big(\frac{\cro{k}^{-1} }{i(\tau-k^3)+k^2+1}(i(\tau-k^3)+k^2+1)
\varphi_N\widetilde{u}\Big)\Big\|_{L^\infty_{t}L^2_{k}}\\
&\lesssim \Big\|\F_t^{-1}\Big(\frac{\cro{k}^{-1}\varphi_N(k)}{i(\tau-k^3)+k^2+1}\Big)\Big\|_{L^\infty_t L^\infty_{k}}\|P_Nu\|_{Y^{0,\frac 12}}\\
&\lesssim \| \cro{k}^{-1}\varphi_N(k)\|_{L^\infty_{k} } \|e^{-t\cro{N}^2}\chi_{\R^+}(t)\|_{L^\infty_t}\|P_Nu\|_{Y^{0,\frac 12}}\\
&\lesssim \cro{N}^{-1}\|P_Nu\|_{Y^{0,\frac 12}}\lesssim \|P_Nu\|_{Y^{-1,\frac 12}}.
\end{align*}
This completes the proof  of \re{est-L2S-11} after square summing in $ N $.
\item
 In the same way,  for any dyadic $ N $
\begin{eqnarray}
\|P_N u\|_{L^2} &\lesssim & \sum_{L\le N^3} \|P_NQ_L u\|_{L^2}  +\sum_{L>N^3}\|P_NQ_L u\|_{L^2} \nonumber \\
& \lesssim &  \sum_{L\le N^3} \|P_N Q_L u\|_{L^2} + \sum_{L>N^3}  L^{1/2} N^{-3/2} \|P_N Q_L u\|_{L^2}\nonumber\\
& \lesssim  & \|P_N u \|_{X^{-1,\frac 12,1}_{1/2} }\label{tf} \; .
\end{eqnarray}
On the other hand, applying  Young and H\"{o}lder's inequalities, we get
\begin{align*}
\|P_Nu\|_{L^2} &= \Big\|\F_t^{-1}\Big(\frac{1}{i(\tau-k^3)+k^2+1}(i(\tau-k^3)+k^2+1)\varphi_N\widetilde{u}\Big)\Big\|_{L^2_{tk}}\\
&\lesssim \Big\|\F_t^{-1}\Big(\frac{\varphi_N(k)}{i(\tau-k^3)+k^2+1}\Big)\Big\|_{L^2_t L^\infty_{k}}\|P_Nu\|_{Y^{0,\frac 12}}\\
&\lesssim \|e^{-t\cro{N}^2}\chi_{\R^+}(t)\|_{L^2_t}\|P_Nu\|_{Y^{0,\frac 12}}\\
&\lesssim \cro{N}^{-1}\|P_Nu\|_{Y^{0,\frac 12}}\lesssim \|P_Nu\|_{Y^{-1,\frac 12}}.
\end{align*}
This proves  \re{est-L2S-1}  after square summing in $ N $.\item
Setting  $v=(\partial_t+\partial_{xxx})u$, we see that $u$ can be rewritten as
$$u(t)=e^{-t\partial_{xxx}}u(0)+\int_0^te^{-(t-t')\partial_{xxx}}v(t')dt'.$$
By virtue of Lemma \ref{lem-Xinfty}, we have
$$\Big(\sum_L[L^{1/2}\|Q_Le^{-t\partial_{xxx}}u(0)\|_{L^2}]^2\Big)^{1/2}\lesssim \|u(0)\|_{L^2}\lesssim \|u\|_{L^\infty_tL^2_x}.$$
Moreover, we get as previously
$$\|u\|_{L^\infty_tL^2_x}\lesssim \Big\|\F_t^{-1}\Big(\frac{1}{i(\tau-k^3)+k^2+1}\Big)\Big\|_{L^\infty_{tk}}\|u\|_{Y^{0,\frac 12}}
\lesssim \|u\|_{Y^{0,\frac 12}}.$$
Now it remains to show that
\begin{equation}\label{est-vL1L2}\Big(\sum_L\Big[L^{1/2}\Big\|Q_L\int_0^te^{-(t-t')\partial_{xxx}}v(t')dt'\Big\|_{L^2}\Big]^2\Big)^{1/2}\lesssim \|v\|_{L^1_tL^2_x},\end{equation}
since the right-hand side is controlled by
\begin{align*}
\|P_Nv\|_{L^1_tL^2_x} &\lesssim \|(I-\partial_{xx})P_Nu\|_{L^1_tL^2_x}+\|P_Nu\|_{Y^{0,\frac 12}}\\
&\lesssim \Big(\Big\|\F_t^{-1}\Big(\frac{\varphi_N(\xi)(\xi^2+1)}{i(\tau-\xi^3)+\xi^2+1}\Big)\Big\|_{L^1_tL^\infty_\xi}+1 \Big)\|P_Nu\|_{Y^{0,\frac 12}}\\
&\lesssim \|P_Nu\|_{Y^{0,\frac 12}}.
\end{align*}
In order to prove \re{est-vL1L2}, we split the integral $\int_0^t = \int_{-\infty}^t-\int_{-\infty}^0$. By Lemma \ref{lem-Xinfty}, the contribution with integrand on $(-\infty,0)$ is bounded by
$$
\lesssim \Big\|\int_{-\infty}^0 e^{t'\partial_{xxx}}v(t')dt'\Big\|_{L^2_x}
\lesssim \|v\|_{L^1_tL^2_x}.
$$
For the last term, we reduce by Minkowski to show that
$$\Big(\sum_L[L^{1/2}\|Q_L(\chi_{t>t'}e^{-(t-t')\partial_{xxx}}v(t'))\|_{L^2_{tx}}]^2\Big)^{1/2}\lesssim \|v(t')\|_{L^2_x}.$$
This can be proved by a time-restriction argument. Indeed, for any $T>0$, we have
\begin{align*}
&\Big(\sum_L[L^{1/2}\|Q_L(\eta_T(t)\chi_{t>t'}e^{-(t-t')\partial_{xxx}}v(t'))\|_{L^2}]^2\Big)^{1/2}\\
&\quad \lesssim \||\tau|^{1/2} \widehat{v}(t')\F_t(\eta_T(t)\chi_{t>t'})(\tau)\|_{L^2}\\
&\quad\lesssim \|v(t')\|_{L^2}\||\tau|^{1/2}\F_t(\eta(t)\chi_{tT>t'})\|_{L^2}\\
&\quad\lesssim \|v(t')\|_{L^2}.
\end{align*}
We conclude by passing to the limit $T\to\infty$.
\end{enumerate}
\end{proof}

\section{Linear estimates}\label{sec-lin}
It is straightforward to check that  estimates on the linear operator $ W(t) $ and on the
 extension of the Duhamel term proven in \cite{MV1} on $ \R $  still hold on $ \T $.  We thus will concentrate ourselves
  on the $ X^{s,\frac 12,1}_\eps \cap  \widetilde{L^\infty_{t} H^{-1}} $ component.

\begin{proposition}
\label{linearL}
\begin{enumerate}
\item  For any $ \eps \ge 0 $ and all $\phi\in H^{-1}(\T)$, we have
\begin{equation}\label{est-lin}\|\eta(t) W(t)\phi\|_{ \widetilde{\S^{-1}_\varepsilon} }\lesssim \|\phi\|_{H^{-1}}.\end{equation}
\item  Let $\L : f\to \L f$ denote the linear operator
\begin{eqnarray}
\L f(t,x) &=& \eta(t)\Bigl( \chi_{\R^+}(t)\int_0^tW(t-t',t-t')f(t')dt'\nonumber \\
 & & + \chi_{\R^-}(t)\int_0^tW(t-t',t+t')f(t')dt'\Bigr) \label{eq-L} \; .
\end{eqnarray}
 If $f\in \NN^{-1}_\eps $ with $ \eps\ge 0$, then
\begin{equation}\label{est-linNhom}\| \L f \|_{\widetilde{\S^{-1}_\varepsilon}} \lesssim \|f\|_{\NN^{-1}_\varepsilon}.\end{equation}
\end{enumerate}
\end{proposition}

\begin{proof}
The first assertion is  a direct consequence of the corresponding  estimate in $ X^{s,\frac 12,1} $ proven in \cite{MV1} together with the continuous embedding $  X^{s,\frac 12,1}\hookrightarrow \widetilde{L^\infty_{t} H^{-1}}\cap X^{s,\frac 12,1}_\varepsilon $ for $ \varepsilon\ge 0$.

To prove the second assertion,  it clearly suffices to show the three  following inequalities
\begin{eqnarray}\label{est-lin3}\|\L f\|_{Y^{0,\frac 12}}&\lesssim &\|f\|_{Y^{0,-\frac 12}},\\
\label{est-lin2}\|\L f\|_{X^{0,\frac 12,1}_\varepsilon}&\lesssim& \|f\|_{X^{0,-\frac 12,1}_\varepsilon}+\|f\|_{Z^{0,-\frac 12}},\\
\label{est-lin4}\|\L f\|_{ \widetilde{L^\infty_{t} H^{-1}}}&\lesssim& \|f\|_{Z^{0,-\frac 12}}.\end{eqnarray}
Estimate \re{est-lin3} has been proved  in \cite{MV1}. To prove \re{est-lin2} we first note that according to \cite{MV1}, it holds
\begin{equation}\label{est-lin1}\|\L f\|_{X^{s,\frac 12,1}}\lesssim \|f\|_{X^{s,-\frac 12,1}}.\end{equation}
It is then not too hard to be convinced that \re{est-lin2} is a consequence of the following estimate:
\begin{multline}
\Bigl(\sum_{N\ge 4}\Bigl[\sum_{L\le N^3}\langle L+N^2\rangle ^{1/2}\|P_N Q_L(\L(Q_{\ge 2N^3}f))\|_{L^2}\Bigr]^2 \Bigr)^{1/2}\\
 \lesssim\|f\|_{Z^{0,-\frac 12}}+\Bigl(\sum_{N\ge 4}\Bigl[\sum_{L= N^3/2}^{N^3} \langle L+N^2\rangle^{-1/2} \|P_N Q_{L} f\|_{L^2}\Bigr]^2 \Bigr)^{1/2}.\label{est-lin5}
\end{multline}
 To prove \re{est-lin4} and \re{est-lin5}  we proceed as in \cite{MV1}. Using the $x$-Fourier expansion
 and setting $ w(t)=U(-t) f(t) $ it is easy to derive that
 $$\L f(t,x) = U(t)\left[\eta(t)\sum_{k\in \Z} e^{ix k }\int_{\R}
\frac{e^{it\tau}e^{(t-|t|)k^2}-e^{-|t| k ^2}}{i\tau+k^2}\tilde{w}(\tau,k)d\tau \right].$$
In particular by  Plancherel and Minkowski,
$$
\|\L f\|_{ \widetilde{L^\infty_{t} L^2_{x}}}=\|U(-t) \L f\|_{ \widetilde{L^\infty_{t} L^2_{x}}}
\le  \Bigl( \sum_{N} \Bigl\|  \eta(t)\varphi_N(k)\int_\R\frac{e^{it\tau}e^{(t-|t|)k^2}-e^{-|t|k^2}}{i\tau+k^2} \tilde{w}(\tau,k) d\tau
  \Bigl\|_{L^2_{k}L^\infty_{t}}^2 \Bigr)^{1/2} .
$$
Now for $ k\in \Z $ fixed and $ v\in {\cal S}(\R) $ we set
$$K_k(v)(t) = \eta(t)\varphi_N(k)\int_\R\frac{e^{it\tau}e^{(t-|t|)k^2}-e^{-|t|k^2}}{i\tau+k^2} v (\tau)d\tau.$$
We thus are  reduced to show that for any $ k\in \Z $,
\begin{equation}\label{est-lin5b}\|K_k(v)\|_{L^\infty_t}\lesssim \Big\|\frac{\varphi_N(k)v(\tau)}{\langle i\tau+k^2\rangle} \Big\|_{L^1_\tau}\end{equation}
and
\begin{multline}\label{est-lin6}
\sum_{L\le N^3}\langle L+N^2\rangle ^{1/2}\| P_L(K_k (    \Phi_{\ge 2N^3}   v) )\|_{L^2_t}\\
\lesssim \Big\|\frac{\varphi_N(k)v(\tau)}{\langle i\tau+k^2\rangle} \Big\|_{L^1_\tau}+\sum_{L=N^3/2}^{N^3} \varphi_N(k) \langle L+ N^2\rangle ^{-1/2}
\| \varphi_{L}  v  \|_{L^2_\tau}.
 \end{multline}
 where we set $ \Phi_{\ge 2N^3}:= \sum_{L\ge 2 N^3} \varphi_L $. 
To prove \re{est-lin5b} it suffices to notice that
$$
\|K_k(v)\|_{L^\infty_t}\lesssim  \|\eta \|_{L^\infty} \varphi_N(k)\int_{\R}\frac{|v(\tau)|}{|i\tau+k^2|}\, d\tau
$$
which gives the result for $ k\neq 0 $. In the case $ k=0 $ we use a Taylor expansion to get
\begin{eqnarray*}
\|K_0(v)\|_{L^\infty_t} & \lesssim & \varphi_N(k) \Bigl[ \|\eta\|_{L^\infty} \int_{|\tau|\ge 1 }  \frac{|v (\tau)|}{|\tau| }\, d\tau +
 \sum_{n\ge 1}\frac{1}{n!}\|t^n\eta\|_{L^\infty}  \int_{|\tau|\le 1 }  \frac{|\tau|^n |v(\tau)|}{|\tau| }\, d\tau \Bigr] \\
  & \lesssim &  \varphi_N(k) \int_{\R }  \frac{|v (\tau)|}{\langle \tau\rangle }\, d\tau   \sum_{n\ge 0}\frac{1}{n!}\|t^n \eta\|_{L^\infty}
\end{eqnarray*}
which is acceptable since $ \|t^n \eta\|_{L^\infty}\lesssim 2^n$. \\
 To get \re{est-lin6}  we first rewrite $ K_k(\Phi_{\geq 2N^3} v) $ as
 $$
 \eta(t)e^{(t-|t|)k^2}\varphi_N(k)\int_{\R}\frac{e^{it\tau}}{i\tau+k^2} \Phi_{\ge 2N^3} (\tau) v(\tau)d\tau - \eta(t)\varphi_N(k)\int_{\T}\frac{e^{-|t|k^2}}{i\tau+k^2}
\Phi_{\ge 2N^3}(\tau)  v(\tau)d\tau.
 $$
 The contribution of the second term  is easily controlled by the first term of the right-hand side of \re{est-lin6} since, according to \cite{MV1},
 $$
 \sum_L\cro{L+N^2}^{1/2}\|\varphi_N(k)P_L(\eta(t)e^{-|t| k^2})\|_{ L^2_t}\lesssim 1.
 $$
  To treat the contribution of the first one, we set $ \theta(t)=\eta(t)e^{(t-|t|)k^2}$  and rewrite this contribution as $ \sum_{L\le N^3} I_L $ with
  \begin{equation} \label{za}
 I_L:=\langle L+N^2\rangle ^{1/2}\Big\| \varphi_N(k) \varphi_L(\tau) \Bigl(
 \widehat{\theta}(\tau')\star [\Phi_{\ge 2N^3}(\tau') \frac{v(\tau')}{i\tau'+k^2} ]\Bigl)(\tau)\Big\|_{L^2_\tau}.
  \end{equation}
  For $ L\le N^3/4 $,  by support considerations we may replace $ \widehat{\theta}(\tau')$ by $\chi_{|\tau'|\ge \frac{N^3}{2}} \widehat{\theta}(\tau') $ in \re{za}.
 Since it is not too hard to check that  two integrations by parts yield
  $|\hat{\theta}(\tau)|\lesssim\frac{\cro{k}^2}{|\tau|^2}$, this ensures that
  \begin{eqnarray*}
\sum_{L\le N^3/4}  I_L &\lesssim & \sum_{L\le N^3/4}\langle L+N^2\rangle ^{1/2} \varphi_N(k)\| \chi_{|\tau|\ge \frac{N^3}{2}} \widehat{\theta}\|_{L^2}  \Big\|\frac{v}{\langle i\tau+k^2\rangle}\Big\|_{L^1}\\
  &\lesssim &  \sum_{L\le N^3/4}\langle L+N^2\rangle ^{1/2}N^{-5/2}\varphi_N(k)
   \Big\|\frac{v}{\langle i\tau+k^2\rangle}\Big\|_{L^1}\\
  & \lesssim & \varphi_N(k)\Big\|\frac{v}{\langle i\tau+k^2\rangle}\Big\|_{L^1}.
  \end{eqnarray*}
  Now, for $ L= N^3$ (Note that  the case $ L=N^3/2 $ can be treated in exactly the same way), we use that $ \varphi_L \equiv \eta(\cdot/2L)  \varphi_L  $ and that by the mean value theorem,
  $
  |\varphi_L(\tau)-\varphi_L(\tau') |\lesssim L^{-1} |\tau-\tau'|
   $. Substituting this in \re{za} we infer that
   \begin{align*}
   I_{N^3} &\lesssim \langle N^3+N^2\rangle ^{1/2}\varphi_N(k)\Bigl\|\int_{\R} \widehat{\theta}(\tau-\tau')\varphi_{N^3}(\tau') \frac{\Phi_{\geq 2N^3}(\tau')v(\tau')}{i\tau'+k^2}d\tau'
\Bigr\|_{L^2_\tau}\\
 &\quad +\langle N^3+N^2\rangle ^{1/2}\varphi_N(k) \Bigl\|\eta(\tau/4N^3)  N^{-3} \int_{\R} |\widehat{\theta}(\tau-\tau')||\tau-\tau'| \frac{|v(\tau')|}{|i\tau'+k^2|} \, d\tau'\Bigr\|_{L^2_\tau}\\
&:=I_{N^3}^1+I_{N^3}^2.
   \end{align*}
   Applying  Plancherel theorem, H\"{o}lder inequality in $t$  and then Parseval theorem,  the first term can be easily  estimated by
  \begin{eqnarray*}
I^1_{N^3}& \lesssim & \langle N^3+N^2\rangle ^{1/2} \varphi_N(k)\| \theta\|_{L^\infty}\Bigl\| \frac{\varphi_{N^3}(\tau) v(\tau)}{i\tau+k^2} \Bigr\|_{L^2} \\
& \lesssim &   \langle N^3\rangle ^{-1/2}  \varphi_N(k)\| \varphi_{N^3}(\tau) v(\tau) \|_{L^2}
   \end{eqnarray*}
   which is acceptable. Finally, note that  $ \|\widehat{\theta}\|_{L^\infty} \le \|\theta\|_{L^1} \le \| \eta\|_{L^1} \lesssim 1$ and that
 integrating by parts one time, it is not too hard to check that  $ |\widehat{\theta}(\tau)|\lesssim \frac{1}{\cro{\tau}}$. This ensures that the second term can be controlled by
  \begin{align*}
  I_{N^3}& \lesssim   \langle N^3\rangle ^{1/2} \varphi_N(k) \Bigl\|\eta(\tau/4N^3)  N^{-3} \int_{\R}\frac{|\tau-\tau'|}
  {\cro{\tau-\tau'}}  \frac{|v(\tau')|}{|i\tau'+k^2|} \, d\tau'  \Bigr\|_{L^2} \\
  & \lesssim  N^{-3/2}  \varphi_N(k)   \|\eta(\tau/4N^3) \|_{L^2_\tau} \Bigl\| \frac{v(\tau)}{i\tau+k^2} \Bigr\|_{L^1_\tau}
  \\
& \lesssim   \varphi_N(k)\Bigl\| \frac{v(\tau)}{i\tau+k^2} \Bigr\|_{L^1_\tau}.
  \end{align*}
   \end{proof}

\section{Bilinear estimate}\label{sectionbilinear}
In this section we provide a proof of the following crucial bilinear estimate.
\begin{proposition}\label{propo}
Let $ 0<\eps<1/12$. Then for all $u,v\in\S^{-1}_\eps$ it holds
\begin{equation}\label{est-bil}\|\partial_x(uv)\|_{\NN^{-1}_\varepsilon}\lesssim \|u\|_{\S^{-1}_\varepsilon}\|v\|_{\S^{-1}_\varepsilon}.\end{equation}
\end{proposition}
We will need the  following sharp estimates  proved in \cite{Tao2}.
 \begin{lemma} \label{lem2}
Let $ u_1 $ and $ u_2 $ be two real valued $ L^2 $ functions defined on $ \R\times \Z $ with the following support properties
\begin{equation*}
(\tau,k)\in \supp u_i \Rightarrow |k|\sim N_i ,
\langle \tau-k^3 \rangle \sim L_i , \, i=1,2.
\end{equation*}
Then the following estimates hold:
\begin{equation*} \label{pro1est1}
\| u_1\star u_2\|_{L^2_\tau L^2(|k|\ge N)} \lesssim \min(L_1, L_2)^{1/2}  \Bigl(  \frac{\max(L_1, L_2)^{1/4}}{N^{1/4}} +1\Bigr)  \| u_1\|_{L^2}  \|  u_2\|_{L^2}
\end{equation*}
and
 if $N_1\gg N_2 $,
\begin{equation*} \label{pro1est2}
\| u_1\star u_2\|_{L^2} \lesssim \min(L_1, L_2)^{1/2} \Bigl(  \frac{\max(L_1, L_2)^{1/2}}{N_1} +1\Bigr)  \| u_1\|_{L^2}  \|  u_2\|_{L^2}.
\end{equation*}
\end{lemma}
{\it Proof of Proposition \ref{propo}.}
First we remark that because of the $L^2_k$ structure of the spaces involved in our analysis we have the following localization property
$$\|f\|_{\S^{-1}_\varepsilon}\sim \Big(\sum_N\|P_Nf\|_{\S^{-1}_\varepsilon}^2\Big)^{1/2}\quad\textrm{and}\quad \|f\|_{\NN^{-1}_\varepsilon}\sim \Big(\sum_N\|P_Nf\|_{\NN^{-1}_\varepsilon}^2\Big)^{1/2}.$$
Performing a dyadic decomposition for $u, v$ we thus obtain
\begin{equation}\label{eq-dec}\|\partial_x(uv)\|_{\NN^{-1}_\varepsilon}\sim \Big(\sum_N\Big\|\sum_{N_1,N_2}P_N\partial_x(P_{N_1}uP_{N_2}v)\Big\|_{\NN^{-1}_\varepsilon}^2\Big)^{1/2}.\end{equation}
We can now reduce the number of case to analyze by noting that the right-hand side vanishes unless one of the following cases holds:
\begin{list}{$\bullet$}
\item (high-low interaction) $N\sim N_2$ and $N_1\lesssim N$,
\item \item (low-high interaction) $N\sim N_1$ and $N_2\lesssim N$,
\item (high-high interaction) $N\ll N_1\sim N_2$.
\end{list}
\noindent
The former two cases are symmetric. In the first case, we can rewrite the right-hand side of (\ref{eq-dec}) as
$$\|\partial_x(uv)\|_{\NN^{-1}_\varepsilon}\sim \Big(\sum_N\|P_N\partial_x(P_{\lesssim N}u P_Nv)\|_{\NN^{-1}_\varepsilon}^2\Big)^{1/2},$$
and it suffices to prove the high-low estimate
\begin{equation}\label{HL}\tag{HL}\|P_N\partial_x(P_{\lesssim N}uP_Nv)\|_{\NN^{-1}_\varepsilon}\lesssim \|u\|_{\S^{-1}_\varepsilon}\|P_Nv\|_{\S^{-1}_\varepsilon}\end{equation}
for any dyadic $N$. If we consider now the third case, we easily get
$$\|\partial_x(uv)\|_{\NN^{-1}_\varepsilon}\lesssim \sum_{N_1}\|P_{\ll N_1}\partial_x(P_{N_1}uP_{N_1}v)\|_{\NN^{-1}_\varepsilon},$$
and it suffices to prove for any $N_1$ the high-high estimate
\begin{equation}\label{HH}\tag{HH}\|P_{\ll N_1}\partial_x(P_{N_1}uP_{N_1}v)\|_{\NN^{-1}_\varepsilon}
\lesssim \|P_{N_1}u\|_{\S^{-1}_\eps}\|P_{N_1}v\|_{\S^{-1}_\varepsilon}\end{equation}
since the claim follows then from Cauchy-Schwarz.\\
Finally, since the $ \S^{-1}_\varepsilon$-norm only sees the size of the modulus of the space-time Fourier transform
 we can always assume that our functions have  real-valued non negative space-time  Fourier transform.\\
Before starting to estimate the different terms we recall the resonance relation associated with the KdV equation that reads
\begin{equation} \label{resonance}
(\tau_1-k_1^3)+(\tau_2-k_2^3)+(\tau_3-k_3^3)=3k_1k_2k_3 \mbox{ whenever } (\tau_1,k_1)+ (\tau_2,k_2)+ (\tau_3,k_3)=0.
\end{equation}

\subsection{High-Low interactions}
We decompose the bilinear term as $$P_N\partial_x(P_{\lesssim N}uP_Nv) = \sum_{N_1\lesssim N}\sum_{L, L_1, L_2}P_NQ_L\partial_x(P_{N_1}Q_{L_1}uP_NQ_{L_2}v).$$
Note first that we can always assume that  $ N_1\gg 1 $, since otherwise, by using Sobolev inequalities and \re{est-L2S-1}, it holds
\begin{eqnarray}
\sum_{N_1\lesssim 1}\|P_N\partial_x(P_{N_1}uP_N v)\|_{Y^{-1,-\frac 12}} &\lesssim & \sum_{N_1\lesssim 1 }\cro{N}^{-1}N\|P_N(P_{N_1}uP_N v)\|_{L^1_tL^2_x}\nonumber \\
&\lesssim & \sum_{N_1\lesssim 1}\|P_{N_1}u\|_{L^2_tL^\infty_x}\|P_N  v\|_{L^2}\nonumber \\
&\lesssim &  \sum_{N_1\lesssim 1 }N_1^{1/2}\|P_{N_1}u\|_{L^2}\|P_N  v\|_{L^2}\nonumber \\
&\lesssim & \|u\|_{\S^{-1}_\eps}\|P_Nv\|_{\S^{-1}_\eps}  \label{uhuh}
\end{eqnarray}
as soon as  $ \varepsilon\le 1/2 $.
We now separate different regions.  It is worth noticing that \re{resonance} ensures that
 $ \max(L,L_1,L_2) \gtrsim N^2 N_1 $.
\subsubsection{ $L\gtrsim N^2N_1$}
We set $ L\sim 2^l N^2 N_1 $. Taking advantage of the $X^{-1,-\frac 12,1}_\eps\cap Z^{-1,-\frac 12}$ part of $\NN^{-1}_\eps$ as well as the continuous embedding $X^{-1,-\frac 12,1} \hookrightarrow X^{-1,-\frac 12,1}_\eps\cap Z^{-1,-\frac 12}$,  by using  Lemma \ref{lem2} we get
\begin{align*}
I_1:&=\sum_{1\ll N_1\lesssim N}\sum_{\substack{l\ge 0\\L_1, L_2 }}\|P_NQ_L\partial_x(P_{N_1}Q_{L_1}uP_N Q_{L_2} v)\|_{X^{-1,-\frac 12,1}}\\
& \lesssim \sum_{1\ll N_1\lesssim N}\sum_{\substack{l\ge 0\\L_1, L_2 }} 2^{-l/2} N^{-1} N_1^{-1/2}
\|P_N Q_L(P_{N_1}Q_{L_1}uP_N Q_{L_2} v)\|_{L^2}\\
& \lesssim \sum_{1\ll N_1\lesssim N}\sum_{\substack{l\ge 0\\L_1, L_2 }}
 2^{-l/2} N^{-1} N_1^{-1/2}
(L_1\wedge L_2)^{1/2} \Bigl( \frac{(L_1\vee L_2)^{1/4}}{N^{1/4}} +1 \Bigr)\\
&\quad \times\| P_{N_1} Q_{L_1} u\|_{L^2}
 \| P_N Q_{L_2} v\|_{L^2}
 \end{align*}
 Noticing that for any $0<\alpha<1$  it holds
$$(L_1\wedge L_2)^{1/2} \Bigl( \frac{(L_{1}\vee L_2)^{1/4}}{N^{1/4}} +1 \Bigr) \lesssim (L_{1}\vee L_2)^{-\frac{\alpha}{4}}
 N^{-\frac{3-2\alpha}{4}}(L_{1}+N_{1}^2)^{1/2}(L_2+N^2)^{1/2},$$
we deduce that
 $$
I_1 \lesssim \sum_{1\ll N_1\lesssim N}\sum_{\substack{l\ge 0\\L_1, L_2 }} 2^{-l/2}  (L_{1} L_2)^{-\alpha/8}
 \Bigl( \frac{ N_{1}}{N}\Bigl)^{\frac{5-2\alpha}{8}}
 \| P_{N_1} Q_{L_1} u\|_{X^{-\frac{9-2\alpha}{8},\frac 12,1}}
 \| P_{N} Q_{L_2} v\|_{X^{-\frac{9-2\alpha}{8},\frac 12,1}}
$$
Taking $\alpha >0$ small enough this proves with \re{est-L2l2} that
$$
I_1\lesssim  \|u\|_{\S^{-1}_\eps}\|P_{N}v\|_{\S^{-1}_\eps}
$$
whenever $ \eps<1/8 $.
\subsubsection{$L_1\gtrsim N^2 N_1$ and  $ L\ll N^2 N_1 $}\label{zz}
We can set $ L_1\sim 2^l N^2 N_1 $ with $l\geq 0$. By duality, it is equivalent to show that \footnote{The space
$  X^{1,\frac 12, \infty} $ is endowed with the norm
$$
\|u\|_{X^{1,\frac 12, \infty}}:=\Big(\sum_N\sup_{{L}}\Big[\cro{N}^{s}\cro{L+N^2}^{b}\|P_NQ_Lu\|_{L^2_{xt}}\Big]^{2}\Big)^{1/2}
$$}
$$I_2\lesssim \|u\|_{\S^{-1}_\eps}\|P_Nv\|_{\S^{-1}_\eps}\|P_Nw\|_{X^{1,\frac 12, \infty}}$$
where
\begin{eqnarray*}
I_2& :=&\sum_{1 \ll N_1\lesssim N}\sum_{\substack{l\ge 0\\L_2, L\ll N^2 N_1 }}
\Bigl| \Bigl( P_NQ_L w \, ,\, \partial_x(P_{N_1}Q_{2^l N^2 N_1 }u P_N Q_{L_2} v)\Bigr)_{L^2} \Bigr|\\
&=&\sum_{1\ll  N_1\lesssim N}\sum_{\substack{l\ge 0\\L_2, L\ll N^2 N_1 }}
\Bigl| \Bigl( \widetilde{P_{N_1}Q_{2^l N^2 N_1 }u}\, ,\, \widetilde{\partial_x P_NQ_L w} \star \check{\widetilde{P_N Q_{L_2} v }})\Bigr)_{L^2} \Bigr|
\end{eqnarray*}
and $ {\check \theta}(\tau,k)=\theta(-\tau,-k) $. According to Lemma \ref{lem2} we get
\begin{eqnarray*}
I_2& \lesssim&\sum_{1\ll  N_1\lesssim N}\sum_{\substack{l\ge 0\\L_2, L\lesssim N^2 N_1 }} 2^{-l/2}
N^{-1} N_1^{-1/2} (L_1^{1/2} \|P_{N_1}Q_{2^l N^2 N_1 }u\|_{L^2})\Bigl\| \widetilde{\partial_x P_N Q_L w}\star  \check{ \widetilde{ P_N Q_{L_2} v}}\Bigr\|_{L^2}\\
& \lesssim&\sum_{1\ll  N_1\lesssim N}\sum_{\substack{l\ge 0\\L_2, L\lesssim N^2 N_1 }} 2^{-l/2}
N^{-1} N_1^{-1/2} (L_1^{1/2} \|P_{N_1}Q_{2^l N^2 N_1 }u\|_{L^2}) \\
& &\times
(L\wedge L_2)^{1/2} \Bigl( \frac{(L\vee L_2)^{1/4}}{N_1^{1/4}} +1 \Bigr)
\| \partial_x P_{N} Q_{L} w\|_{L^2}
 \| P_N Q_{L_2} v\|_{L^2}.
\end{eqnarray*}
Since for any $\eps>0$ we have the estimate
$$(L\wedge L_2)^{1/2} \Bigl( \frac{(L\vee L_2)^{1/4}}{N_1^{1/4}} +1 \Bigr) \lesssim (L\vee L_2)^{-\varepsilon/2} N_1^{-\frac 34+\eps}(L+N^2)^{1/2}(L_2+N^2)^{1/2},$$
it follows that
\begin{eqnarray*}
 I_2& \lesssim&\sum_{1\ll  N_1\lesssim N}\sum_{\substack{l\ge 0\\L_2, L\lesssim N^2 N_1 }}2^{-l/2}
  (L L_2)^{-\varepsilon/4}
N_1^{-\frac{1}{4}+3\varepsilon}   \\
& & \times\|P_{N_1}Q_{2^l N^2 N_1 }u\|_{X^{-1-\varepsilon,\frac 12,1}}
\|  P_{N} Q_{L} w\|_{X^{1,\frac 12,\infty}}
  \| P_N Q_{L_2} v\|_{X^{-1-\varepsilon,\frac 12,1}}
\end{eqnarray*}
which is acceptable whenever $ \varepsilon<1/12$.
\subsubsection{$L_2\gtrsim N^2 N_1$ and $  \, L\vee L_1 \ll N^2 N_1$}
In this region thanks to the resonance relation \re{resonance} one has  $ L_2\sim N^2 N_1 $. We proceed as in the preceding subsection.
We get
\begin{align}
I_3& :=
\sum_{1\ll N_1\lesssim N}\sum_{L\vee L_1\ll N^2 N_1 }
 \Bigl| \Bigl( P_NQ_L w \, ,\, \partial_x(P_{N}Q_{ N^2 N_1 }v P_{N_1} Q_{L_1} u)\Bigr)_{L^2} \Bigr|\nonumber\\
& \lesssim \sum_{1\ll N_1\lesssim N}\sum_{L\vee L_1\lesssim N^2 N_1 }
N^{-1} N_1^{-1/2} (L_2^{1/2} \|P_{N}Q_{ N^2 N_1 }v\|_{L^2}) \nonumber\\
&
\quad \times(L\wedge L_1)^{1/2} \Bigl( \frac{(L\vee L_1)^{1/4}}{N^{1/4}} +1 \Bigr)
\| \partial_x P_{N} Q_{L} w\|_{L^2}
 \| P_{N_1} Q_{L_1} u\|_{L^2}.\label{est-nl1}
 \end{align}
On the other hand, we clearly have
$$
(L\wedge L_1)^{1/2} \Bigl( \frac{(L\vee L_1)^{1/4}}{N^{1/4}} +1 \Bigr) \lesssim (L\vee L_1)^{-\varepsilon/2} N_1^{-\frac 34+\eps}(L+N^2)^{1/2}(L_1+N_1^2)^{1/2}
$$
Inserting this into (\ref{est-nl1}) we deduce
 \begin{align}
 I_3&  \lesssim\sum_{1\ll N_1\lesssim N}\sum_{L\vee L_1\lesssim N^2 N_1 }
  (L L_1)^{-\varepsilon/4}
N_1^{-\frac{1}{4}+2\varepsilon}   \|P_{N}Q_{ N^2 N_1 }v\|_{X^{-1,\frac 12,1}} \nonumber\\
& \quad\times
\|  P_{N} Q_{L} w\|_{X^{1,\frac 12,\infty}}
  \| P_{N_1} Q_{L_1} u\|_{X^{-1-\varepsilon,\frac 12,1}}.\label{ml}
\end{align}
Now either $ N_1\le N $ or $ N_1 \sim N $. In the first case we have
$   \|P_{N}Q_{ N^2 N_1 }v\|_{X^{-1,\frac 12,1}}=   \|P_{N}Q_{ N^2 N_1 }v\|_{X^{-1,\frac 12,1}_\eps}$ which shows that   \re{ml}  is acceptable for $0< \varepsilon<1/8 $.
 In the  second case, we have
 $  N_1^{-\varepsilon}\|P_{N}Q_{ N^2 N_1 }v\|_{X^{-1,\frac 12,1}}\le   \|P_{N}Q_{ N^2 N_1 }v\|_{X^{-1,\frac 12,1}_\eps}$ which shows that
  \re{ml} is acceptable for $0< \varepsilon<1/12 $.
\subsection{High-High interactions}
We perform the decomposition
$$P_{\ll N_1}\partial_x(P_{N_1}uP_{N_1}v) = \sum_{N\ll N_1}\sum_{L,L_1,L_2}P_NQ_L\partial_x(P_{N_1}Q_{L_1}uP_{N_1}Q_{L_2}v).$$
By symmetry we can assume that $L_1\ge L_2 $.  Then \re{resonance} ensures that $ \max(L,L_1) \gtrsim N N_1^2 $.
\subsubsection{ $ N_1^2 N \lesssim L_1 \le N_1^3 $.}
We can set  $L_1\sim 2^l N_1^2N $ with $ l\ge 0 $. Using the $Y^{-1,-\frac 12}$ part of $\NN^{-1}_\eps$, we want to estimate
\arraycolsep3pt
\begin{align*}
I_4&:=  \Big\|\sum_{N\ll N_1\atop l\ge 0 }P_N\partial_x(P_{N_1}Q_{2^l N_1^2N}uP_{N_1}v)\Big\|_{Y^{-1,-\frac 12}}\\
&\lesssim \Big(\sum_{N\ll N_1\atop l\ge 0}\|P_N(P_{N_1}Q_{2^l N_1^2N}uP_{N_1}v)\|_{L^1_tL^2_x}^2\Big)^{1/2}\\
  & \lesssim \Big(\sum_{N\ll N_1\atop l\ge 0}[N^{1/2}\|P_{N_1}Q_{2^l N_1^2N}u\|_{L^2}\|P_{N_1}v\|_{L^2}]^2\Big)^{1/2}.
\end{align*}
According to \re{est-L2S-1} and \re{est-L2l2}, this leads for $ \varepsilon\le 1/2 $ to
\begin{eqnarray*}
I_4 & \lesssim & \Big(\sum_{N\ll N_1\atop l\ge 0}  2^{-l}\sum_{L_1\sim 2^l N^2_1N\le N_1^3}[N_1^{-1}L_1^{1/2}\|P_{N_1}Q_{L_1}u\|_{L^2}]^2\Big)^{1/2}\|P_{N_1}v\|_{\S^{-1}_\varepsilon}\\
& \lesssim & \Big(\sum_{L_1 \le N_1^3 \atop l\ge 0 }   2^{-l}\sum_{N\sim   2^{-l} L_1N_1^{-2}}[N_1^{-1}L_1^{1/2}\|P_{N_1}Q_{L_1}u\|_{L^2}]^2\Big)^{1/2}\|P_{N_1}v\|_{\S^{-1}_\varepsilon}\\
& \lesssim & \Big(\sum_{L_1\le N_1^3} [N_1^{-1}L_1^{1/2}\|P_{N_1}Q_{L_1}u\|_{L^2}]^2\Big)^{1/2}\|P_{N_1}v\|_{\S^{-1}_\varepsilon}\\
& \lesssim & \|P_{N_1}u\|_{\S^{-1}_\varepsilon}\|P_{N_1}v\|_{\S^{-1}_\varepsilon}.
\end{eqnarray*}
\subsubsection{ $ L_1 \ge N_1^3$}\label{523}
We can set $ L_1 = 2^l N_1^3 $ with $ l\ge 0 $.  We proceed  by duality as in Subsection \ref{zz}  to get
$$
I_5 \lesssim\sum_{1\lesssim N\ll N_1}\sum_{\substack{l\ge 0\\L_2, L }}
 2^{-l/2} N_1^{-\frac 12+\eps} (N_1^{-1-\eps}L_1^{1/2} \|P_{N_1}Q_{2^l  N_1^3 }u\|_{L^2})
 \Bigl\| \widetilde{\partial_x P_N Q_L w}\star
  \check{ \widetilde{ P_{N_1} Q_{L_2} v}}\Bigr\|_{L^2}.
$$
By virtue of Lemma \ref{lem2}, we have the bound
\begin{align*}
\Bigl\| \widetilde{\partial_x P_N Q_L w}\star
  \check{ \widetilde{ P_{N_1} Q_{L_2} v}}\Bigr\|_{L^2} &\lesssim (L\wedge L_2)^{1/2} \Bigl( \frac{(L\vee L_2)^{1/4}}{N_1^{1/4}} +1 \Bigr)
\| \partial_x P_{N} Q_{L} w\|_{L^2}
 \| P_{N_1} Q_{L_2} v\|_{L^2}\\
 &\lesssim (L\vee L_2)^{-\varepsilon/2}N_1^{\frac 14+2\eps} \|  P_{N} Q_{L} w\|_{X^{1,\frac 12,\infty}} \| P_{N_1} Q_{L_2} v\|_{X^{-1-\varepsilon,\frac 12,1}}.
\end{align*}
Thus it is enough to check that
$$\sum_{1\lesssim N\ll N_1}\sum_{\substack{l\ge 0\\L_2, L }}2^{-l/2}
  (L L_2)^{-\varepsilon/4}
N_1^{-\frac{1}{4}+3\varepsilon}\lesssim 1,$$
but this is easily verified for $\eps<1/12$.
\subsubsection{ $L\gtrsim N_1^2 N$ and $ N_1^2 N^{1-\varepsilon} \le L_{2}\le L_{1}\ll N_1^2 N$}
Then, by the resonance relation \re{resonance} we must have $ L\sim N_1^2 N$. We  set $ L_2\sim2^q  N_1^2 N^{1-\varepsilon}$ and $ L_1=2^p L_2 $ with $q\ge 0 $ and $ p\ge 0$.
Since $ N\ll N_1 $ we are in the region $ L\gtrsim N^3 $. However since $ X^{-1,-\frac 12,1} \hookrightarrow   X^{-1-\varepsilon,-\frac 12,1}\cap Z^{-1,-\frac 12} $,
it suffices to show that
$$I_6\lesssim \|P_{N_1}u\|_{\S^{-1}_\eps}\|P_{N_1}v\|_{\S^{-1}_\eps}\Big(\sum_N\|P_NQ_{N_1^2N}w\|_{X^{1, \frac 12,\infty}}^2\Big)^{1/2}$$
where
$$
I_6:=\sum_{N\ll N_1} \sum_{p\ge 0,q\ge 0 }\Big|\Big( P_NQ_{N_1^2 N} w,\partial_x ( P_{N_1} Q_{L_1} u P_{N_1}
 Q_{L_2} v) \Big)_{L^2}\Big|.
$$
Using Lemma \ref{lem2} we get
 \begin{align*}
 I_6 &\lesssim  \sum_{N\ll N_1} \sum_{p\ge 0,q\ge 0 } \| P_{N_1} Q_{L_1} u \|_{L^2}
 N_1^{-1} \| P_N Q_{N_1^2 N} w \|_{X^{1,\frac 12,\infty}} \| P_{N_1} Q_{L_2} v \|_{X^{0,\frac 12,1}} \\
 & \lesssim  \sum_{N\ll N_1} \sum_{p\ge 0,q\ge 0 }
   2^{-q/2} 2^{-p/2} N^{-\frac{1}{2}+\frac{\varepsilon}{2}}
    \| P_{N_1} Q_{2^p 2^q N_1^2 N^{1-\varepsilon}} Q_{\le N_1^3} u \|_{X^{-1,\frac 12,1}}
 \\
  & \quad\times \| P_N Q_{N_1^2 N}w \|_{X^{1,\frac 12,\infty}}\| P_{N_1} Q_{2^q  N_1^2 N^{1-\varepsilon}}
   Q_{\le N_1^3}v \|_{X^{-1,\frac 12,1}}
 \end{align*}
 which is acceptable as soon as $ \eps<1$.

\subsubsection{$L\gtrsim N_1^2 N$ ,  $ L_1 \ll N_1^2 N  $ and $ L_{2}\le N_1^2 N^{1-\varepsilon} $}
Since $ N\ll N_1 $ we are in the region $ L> N^3 $. It thus suffices to  estimate  both the $ X^{-1-\varepsilon,-\frac 12,1} $ and the $ Z^{-1,-\frac 12} $ norms. Let us start by estimating the first one.  Note that in this region we can replace $ P_{N_1} u $ and $ P_{N_1} v$ by
 $ P_{N_1}  Q_{\le N_1^3}  u $ and  $ P_{N_1}  Q_{\le N_1^3}  v $. Taking into account the gain of $\eps$ in the definition of the space, we get
\begin{align*}
\| \partial_x(uv )\|_{X^{-1-\varepsilon,-\frac 12,1}}& \lesssim \sum_{1\lesssim N\ll N_1}
 N^{-\varepsilon} N_1^{-1} N^{-1/2}  \|P_N (P_{N_1} Q_{\le N_1^3} u P_{N_1} Q_{\le N_1^3} v)\|_{L^2}\\
 & \lesssim \sum_{1\lesssim N\ll N_1}
 N^{-\varepsilon} N_1^{-1}   \|P_{N_1}Q_{\le N_1^3}u
 \,
 P_{N_1}
Q_{\le N_1^3}  v\|_{L^2_t L^1_x }\\
 & \lesssim \sum_{1\lesssim N\ll N_1}
 N^{-\varepsilon}  (N_1^{-1}
 \|P_{N_1}Q_{\le N_1^3}u\|_{L^\infty_t L^2_x}) \|P_{N_1} Q_{\le N_1^3} v\|_{L^2 },
\end{align*}
which is acceptable as soon as $ \varepsilon>0 $.
It remains to estimate the $ Z^{-1,-\frac 12} $-norm.
 By duality we have to estimate
$$
I_7:=  \sum_{N\ll N_1} N_1^{-2} N^{-1} \Bigl| \Bigl( \widehat{P_N w }\chi_{\langle \sigma\rangle\sim N_1^2 N}
\, , \, \widetilde{ P_{N_1}Q_{\le N_1^3} u}\star
\widetilde{P_{N_1} Q_{ \le N_1^2 N^{1-\varepsilon }} v}  \Bigr)_{L^2}\Bigr|
$$
where $ w$ only depends on $ k$ and with $\sigma=\tau-k^3$ (recall that we can assume that  the space-time  Fourier transforms of $ u $ and $ w$ are
 non negative real-valued functions). We follow an idea that can be found in  \cite{Bo3}.
First we notice that for any fixed $ k $,
$$
\chi_{\langle\sigma \rangle\sim L} \lesssim \chi_{\langle\sigma \rangle\sim L}\star_\tau (\frac{1}{L}  \chi_{\langle\sigma \rangle\le L})
$$
and thus the above scalar product  can be rewritten as
$$
 \Bigl| \Bigl( \widehat{P_N w }
\chi_{\langle \sigma\rangle\sim N_1^2 N} \, ,\,
\widetilde{P_{N_1}Q_{\le N_1^3} u}\star  \widetilde{P_{N_1}
Q_{\le N_1^2 N^{1-\varepsilon}  } v} \star_\tau (\frac{1}{ N_1^2 N}  \chi_{\langle\sigma
\rangle\le  N_1^2 N}) \Bigr)_{L^2}\Bigr|
$$
where $ {\mathcal F}^{-1} \Bigl[\widetilde{P_{N_1} Q_{\le N_1^2 N^{1-\varepsilon }} v}
\star_\tau (\frac{1}{ N_1^2 N}  \chi_{\langle\sigma \rangle\le  N_1^2 N}) \Bigr]$ is of the form $ P_{N_1} Q_{\lesssim  N_1^2 N} v' $ with
\begin{equation}\label{dc}
\| P_{N_1} Q_{\lesssim  N_1^2 N} v' \|_{L^2}
 \lesssim
N^{-\varepsilon/2}
\| P_{N_1}  Q_{ \le N_1^2 N^{1-\varepsilon}  } v \|_{L^2} \; .
\end{equation}
Indeed the linear operator $ T_{K,K_2}\, :\, v\mapsto \frac{1}{K} v (\cdot)  \chi_{\{\langle \cdot \rangle \le  K_2\}}\star \chi_{\{\langle \cdot \rangle \le  K\}}$
 is a continuous endomorphism  of $ L^1(\R) $ and $ L^\infty(\R) $ with
 \begin{eqnarray*}
 \| T_{K_,K_2}v\|_{L^\infty (\R)}  &\le & \sup_{x\in\R} \frac{1}{K} \Bigl|  \int_{\R} v(y) \chi_{\{\langle y \rangle \le   K_2\}} \chi_{\{\langle x-y \rangle \le  K\}}
 \, dy \Bigr| \\
 & \lesssim &  \frac{\min (K,K_2)}{K} \| v \|_{L^\infty(\R)}
 \end{eqnarray*}
 and
 $$
  \| T_{K_,K_2}v\|_{L^1 (\R)}  \le \frac{1}{K} \| v\|_{L^1(\R)} \|  \chi_{\{\langle \cdot \rangle \le  K\}}\|_{L^1(\R)} \lesssim \|v \|_{L^1(\R)} \; .
 $$
 Therefore, by Riesz interpolation theorem $ T_{K,K_2} $ is a continuous endomorphism of $ L^2(\R) $ with
 $$
  \| T_{K_,K_2}v\|_{L^2 (\R)} \lesssim \min\Bigl(1,
\frac{K_2}{K} \Bigr)^{1/2} \| v\|_{L^2(\R)} \; .
 $$
 Hence, by Sobolev in $k $ and \re{dc},
 \begin{eqnarray*}
 I_7 & \lesssim &  \sum_{N\ll N_1} N_1^{-2} N^{-1}  N^{1/2} \| \widehat{P_N w}
\chi_{\langle \sigma\rangle\sim  N_1^2 N}\|_{L^2}
\| P_{N_1} Q_{\le N_1^3} u\|_{L^\infty_t L^2_x} \| P_{N_1} Q_{\lesssim N_1^2 N} v' \|_{L^2}\\
& \lesssim & \sum_{N\ll N_1}   N^{-\varepsilon/2}
\|P_N w \|_{L^2(\T)} (N_1^{-1}  \| P_{N_1} Q_{\le N_1^3} u\|_{L^\infty_t L^2_x}) \| P_{N_1} v\|_{L^2},
\end{eqnarray*}
which is acceptable as soon as $ \varepsilon>0 $.\hfill  $\square$

\section{Well-posedness}\label{sectionwell}
In this section, we prove the well-posedness result. The proof follows exactly the same lines as in \cite{MV1}. Using a standard fixed point procedure, it is clear that the bilinear estimate (\ref{est-bil}) allows us to show local well-posedness but for small initial data only. This is because $H^{-1}$ appears as a critical space for KdV-Burgers.
Indeed, on one hand,  we cannot get any contraction factor by restricting time. On the other hand, a dilation argument does not work here since the reduction of the $ H^{-1} $-norm of the dilated initial data would be  exactly compensated by the diminution of the
 dissipative coefficient in front of $ u_{xx}$ (that we take equal to $ 1 $ in \re{KdVB}) in the equation satisfied by the  dilated solution. In order to remove the size restriction on the data, we  change the metric on our resolution space.

For $ 0<\varepsilon<1/12 $ and $\beta\ge 1 $, let us define the following norm on $ \widetilde{\S^{-1}_\varepsilon} $,
$$\|u\|_{\ZZ_\beta} = \inf_{\substack{u=u_1+u_2\\ u_1\in\widetilde{\S^{-1}_\varepsilon}, u_2\in \widetilde{\S^0_\varepsilon}}}\left\{\|u_1\|_{\widetilde{\S^{-1}_\varepsilon}}+\frac 1\beta \|u_2\|_{\widetilde{\S^0_\varepsilon}}\right\}.$$
Note that this norm is equivalent to $ \| \cdot \|_{\widetilde{\S^{-1}_\varepsilon}} $.   Now we will need the following modification
 of Proposition \ref{propo}. This new proposition means that as soon as we assume more regularity  on $ u$ we can get
  a contractive factor for small times in the bilinear estimate.
 \begin{proposition}
There exists $ \nu>0 $ such that for any  $ 0<\varepsilon<1/12 $ and all $(u,v)\in\S^{0}_\varepsilon\times \S^{-1}_\varepsilon $, with compact support (in time) in $[-T,T]$,
 it holds\
\begin{equation}\label{est-bil3}\|\partial_x(uv)\|_{\NN^{-1}_\varepsilon}\lesssim T^\nu \|u\|_{\S^{0}_\varepsilon}\|v\|_{\S^{-1}_\varepsilon}.
\end{equation}
 \end{proposition}
 \begin{proof}
It suffices to slightly modify  the proof of Proposition \ref{propo} to make use of  the following result that can be found in   [\cite{GTV}, Lemma 3.1] (see
 also  [\cite{MR2}, Lemma 3.6]):
 For any $\theta >0$, there exists $\mu=\mu(\theta)>0$ such that for any smooth function  $ f$ with compact support in
  time in $[-T,T]$,
\begin{equation}
 \left\| {\mathcal F}^{-1}_{t,x}\left( \frac{\hat{f}(\tau,k)} {\langle\tau - k^3\rangle^{\theta}}\right) \right\|_{L^{2}_{t,x}}
 \lesssim
   T^\mu \|f\|_{L^{2,2}_{t,x}} \;.
 \label{strichartz}
 \end{equation}
 According to \re{est-L2l2} this  ensures, in particular,  that  for any $ w\in \S^0_{3/8}$ with compact support in $[-T,T] $  it holds
  \begin{equation}\label{lala}{
 \| w\|_{L^2_t H^{3/4}}\lesssim  \| w\|_{X^{0,\frac 38,2}_{3/8}}}\lesssim
 T^{\mu(\frac{1}{8})}  \| w\|_{X^{0,\frac 12,2}_{3/8}}\lesssim  T^{\mu(\frac{1}{8})} \|w\|_{\S^0_{3/8}} \; .
 \end{equation}
  It is pretty clear that the interactions between high  frequencies of $ u $ and high or low frequencies of $ v$
   can be treated by following   the proof of Proposition \ref{propo}
   and using \re{lala}. The  region that seems the most dangerous is   the one of interactions between low frequencies of
    $ u $ and  high frequencies of $ v$ in  the proof of Proposition \ref{propo}.
     But actually in this region, except in the subregion $N_1\lesssim 1 $, we can notice that we may  keep some powers of $ L_1 $ or $L_2 $ in the estimates and thus \re{lala} ensures that \re{est-bil3} holds (one can even replaced $\S^{0}_\varepsilon$ by $\S^{-1}_\varepsilon$) .  Finally, in the subregion $N_1\lesssim 1 $, \re{est-bil3}
     follows directly by applying  \re{lala} in the next to the last line in \re{uhuh}.
  \end{proof}
We are now in position   to prove that the application
$$F_\phi^T : u\mapsto \eta(t)\Big[W(t)\phi-\frac 12\L\partial_x(\eta_T u)^2\Big],$$
where $\L$ is defined in (\ref{eq-L}), is contractive on a ball of $\ZZ_\beta$ for a suitable $\beta>0 $
 and $ T>0 $ small enough. Assuming this for a while, the  local part of Theorem \ref{wellposed}  follows by using standard arguments. Note that the uniqueness will hold in the restriction spaces
   $ \widetilde{{\mathcal S}^{-1}_\eps}(\tau) $ endowed with the norm
 $$
 \|u\|_{  \widetilde{{\mathcal S}^{-1}_\eps}(\tau)}:=\inf_ {v\in \widetilde{{\mathcal S}^{-1}_\eps}} \{ \|v\|_{{\widetilde{\mathcal S^{-1}_\eps}}}, \; v\equiv u \mbox{ on } [0,\tau] \}\; .
 $$
 Finally, to see that the solution $ u $ can be extended for all positive times and belongs to $C(\R_+^*;H^\infty) $ it suffices to notice that, according to
  \re{est-L2S-1},
  $u\in {\mathcal S}^{-1}_\eps(\tau) \hookrightarrow L^2(]0,\tau[\times \T)$ . Therefore, for any $0<\tau'<\tau $  there exists
  $t_0\in ]0,\tau'[ $, such that $ u(t_0)$ belongs  to $L^2(\T)$ . Since according to   \cite{MR2}, \re{KdVB} is globally  well-posed in $ L^2(\T) $ with a solution belonging to $ C(\R^*_+;H^\infty(\T)) $,  the conclusion follows.

  In order to prove that $ F_\phi^T $ is contractive, the first step is  to establish the following result.
\begin{proposition}\label{prop-bilZ} For any  $ \beta\ge 1 $ there   exists $0<T=T(\beta)<1$ such that for any $u,v\in\ZZ_{\beta}$ with compact support in $[-T,T]$ we have
\begin{equation}\label{est-bilZ}
\|\L\partial_x(uv)\|_{\ZZ_\beta}\lesssim \|u\|_{\ZZ_\beta}\|v\|_{\ZZ_\beta}.
\end{equation}
\end{proposition}
Assume for the moment that (\ref{est-bilZ}) holds and let $u_0\in H^{-1}$ and $\alpha>0$.
Split the data $u_0$ into low and high frequencies:
$$u_0=P_{\lesssim N}u_0+P_{\gg N}u_0$$
for a dyadic number $N$. Taking $N=N(\alpha)$ large enough, it is obvious to check that $\|P_{\gg N}u_0\|_{H^{-1}}\leq \alpha$. Hence, according to \re{est-lin},
$$\|\eta(\cdot)W(\cdot)P_{\gg N}u_0\|_{\ZZ_\beta}\lesssim \alpha.$$
Using now the $\widetilde{\S^0_\varepsilon} $-part of $\ZZ_\beta$, we control the low frequencies as follows:
$$\|\eta(\cdot)W(\cdot)P_{\lesssim N}u_0\|_{\widetilde{\S^0_\varepsilon}} \lesssim \frac 1\beta \|P_{\lesssim N}u_0\|_{L^2}\lesssim \frac{N}{\beta}\|u_0\|_{H^{-1}}.$$
Thus we get
$$\|\eta(\cdot)W(\cdot)P_{\lesssim N}u_0\|_{\ZZ_\beta}\lesssim \alpha\ \textrm{ for }\ \beta\gtrsim\frac{N \|u_0\|_{H^{-1}}}{\alpha}.$$
Since $\alpha$ can be chosen as small as needed, we conclude with (\ref{est-bilZ}) that $F_\phi^T$ is contractive on a ball of $\ZZ_\beta$ of radius $R\sim \alpha$ as soon as $ \beta \gtrsim  N \|u_0\|_{H^{-1}}/\alpha $
 and $T= T(\beta)$.

\begin{proof}[Proof of Proposition \ref{prop-bilZ}] By definition on the function space $ {\mathcal Z}_\beta $,
there exist $u_1,v_1\in\widetilde{\S^{-1}_\varepsilon}$ and $u_2,v_2\in \widetilde{\S^0_\eps} $ such that $u=u_1+u_2$, $v=v_1+v_2$ and
\begin{align*}\|u_1\|_{\widetilde{\S^{-1}_\varepsilon}}+\frac 1\beta\|u_2\|_{\widetilde{\S^0_\varepsilon}}\leq & 2\|u\|_{\ZZ_\beta},\\ \|v_1\|_{\widetilde{\S^{-1}_\varepsilon}}+\frac 1\beta\|v_2\|_{\widetilde{\S^0_\varepsilon}}\leq & 2\|v\|_{\ZZ_\beta}.
\end{align*}
Thus one can decompose the left-hand side of (\ref{est-bilZ}) as
\begin{align*}
\|\L\partial_x(uv)\|_{\ZZ_\beta} &\lesssim \|\L\partial_x(u_1v_1)\|_{\widetilde{\S^{-1}_\varepsilon}}+\|\L\partial_x(u_1v_2+u_2v_1)\|_{\widetilde{\S^{-1}_\varepsilon}}+\|\L\partial_x(u_2v_2)\|_{\widetilde{\S^{-1}_\varepsilon}}\\
&= I+II+III.
\end{align*}
From the estimates (\ref{est-linNhom}) and (\ref{est-bil}) we get
$$I\lesssim  \|\partial_x(u_1v_1)\|_{\NN^{-1}_\varepsilon}\lesssim \|u_1\|_{S^{-1}_\varepsilon}\|v_1\|_{\S^{-1}_\varepsilon}\lesssim \|u\|_{\ZZ_\beta}\|v\|_{\ZZ_\beta}.$$
On the other hand, we obtain from (\ref{est-bil3}) that
$$III\lesssim T^{\nu}\|u_2\|_{\S^0_\varepsilon}\|v_2\|_{\S^0_\varepsilon}\lesssim \beta^2 T^{\nu}\|u\|_{\ZZ_\beta}
\|v\|_{\ZZ_\beta}.$$
and
\begin{eqnarray*}
II &  \lesssim &  T^{\nu}(\|u_1\|_{\S^{-1}_\varepsilon}\|v_2\|_{\S^0_\varepsilon}+\|u_2\|_{\S^0_\varepsilon}\|v_1\|_{\S^{-1}_\varepsilon})\\
&\lesssim &  \beta T^{\nu}\|u\|_{\ZZ_\beta}\|v\|_{\ZZ_\beta}.
\end{eqnarray*}
We thus get
$$\|\L\partial_x(uv)\|_{\ZZ_\beta}\lesssim (1+(\beta+\beta^2) T^\nu)\|u\|_{\ZZ_\beta}\|v\|_{\ZZ_\beta}.$$
This ensures that (\ref{est-bilZ}) holds for $T\sim \beta^{-2/\nu}\le 1 $.
\end{proof}

\begin{center}Luc Molinet, \\
{\small
  Laboratoire de Math\'ematiques et Physique Th\'eorique, Universit\'e Fran\c cois Rabelais Tours, F\'ed\'eration Denis Poisson-CNRS, Parc Grandmont, 37200 Tours, FRANCE.  {\it Luc.Molinet@lmpt.univ-tours.fr}}\vspace*{4mm}\\
St\'ephane Vento\\
{\small  L.A.G.A., Institut Galil\'ee, Universit\'e Paris 13,\\
93430 Villetaneuse, FRANCE. {\it vento@math.univ-paris13.fr}}
\end{center}

\end{document}